\documentclass{amsart}

\usepackage{latexsym,enumerate}
\usepackage{amsmath,amsthm,amsopn,amstext,amscd,amsfonts,amssymb}
\usepackage[ansinew]{inputenc}
\usepackage{verbatim}
\usepackage{graphicx}
\usepackage{pstricks}
\usepackage{wasysym}
\usepackage[all]{xy}
\usepackage{epsfig} % Para grÃ¯Â¿Â½ficos
\usepackage{graphicx,psfrag} % Mejor Ã¯Â¿Â½ste para grÃ¯Â¿Â½ficos: es mÃ¯Â¿Â½s moderno
\usepackage{subfigure}
\usepackage{pstricks,pst-node}
\usepackage{multicol}
\usepackage{color}
\usepackage{colortbl}

\newcommand{\NN}{\mathbb{N}}

\newtheorem{theorem}{Theorem}[section]
\newtheorem{lemma}[theorem]{Lemma}
\newtheorem{proposition}[theorem]{Proposition}
\newtheorem{corollary}[theorem]{Corollary}
\newtheorem{definition}[theorem]{Definition}
\newtheorem{example}[theorem]{Example}
\newtheorem{remark}[theorem]{Remark}

\newcommand{\spb}[1]{\smallskip}
\newcommand{\mpb}[1]{\medskip}
\newcommand{\bpb}[1]{\bigskip}

\newcommand{\g}{\gamma}

\begin{document}
\DeclareGraphicsExtensions{.jpg,.pdf,.mps,.png}

\title{A note on $k$-metric dimensional graphs}

\author[Samuel G. Corregidor]{Samuel G. Corregidor}
\address{ Facultad CC. Matem\'aticas, Universidad Complutense de Madrid,
Plaza de Ciencias, 3. 28040 Madrid, Spain}
\email{samuguti@ucm.es}

\author[\'{A}lvaro Mart\'{\i}nez-P\'erez]{\'{A}lvaro Mart\'{\i}nez-P\'erez$^{(1)}$}
\address{ Facultad CC. Sociales, Universidad de Castilla-La Mancha,
Avda. Real F\'abrica de Seda, s/n. 45600 Talavera de la Reina, Toledo, Spain}
\email{alvaro.martinezperez@uclm.es}
\thanks{$^{(1)}$ Supported in part by a grant
from Ministerio de Econom{\'\i}a y Competitividad (MTM 2015-63612P), Spain.
}

\date{\today}

%\centerline{{\bf k-metric dimension-25.TEX}}

\begin{abstract} Given a graph $G = (V,E)$, a set 
$S \subset V$ is called a $k$-\emph{metric generator} for $G$ if any pair of different vertices of $G$ is distinguished by at least $k$ elements of $S$. 
A graph is $k$-\emph{metric dimensional} if $k$ is the largest integer such
that there exists a $k$-metric generator for $G$. This paper studies some bounds on the number $k$ for which a graph is $k$-metric dimensional.
\end{abstract}

\maketitle{}



\section{Introduction}

The concept of metric dimension of a graph naturally arises in applications. Suppose there is a a graph or network and there is the need to locate something in it using detectors placed on certain vertices. Then, any vertex must be uniquely determined by the distances to the detectors. This notion has been developed independently by J. P. Slater in \cite{S1,S2}, where the sets of vertices able to locate every node are called \emph{locating sets}, and Harary and Melter in \cite{HM}, where these sets are called \emph{resolving sets}. Harary and Melter also coined the name of metric dimension for the cardinality of a minimum resolving set.

This concept has been widely used in many areas as hazard detection in networks, see \cite{KCL,UTS};  navigation of robots in networks, see \cite{KRR,MT} or chemistry, see \cite{CEJO,J1,J2}. 
%or pattern recognition and image processing, see also \cite{MT}\footnote{check}.

Metric dimension has been extensively studied, see for example \cite{ALA,BC,CH,CGH,CEJO,CSZ,FGO,T,YKR,YR} and the references therein.
There are also several natural extensions of the definition of metric dimension in the literature, some of them, combining it with the idea of domination. See, for example 
\cite{HH,OPZ,SZ}

Another natural extension of metric dimension appears in \cite{ERY}. See also \cite{EYR1,EYR2,EYR3,YER}. The idea is that, in order to improve the accuracy of the detection or the robustness of the system, it may be interesting to have a family of detectors such that every pair of vertices is distinguished by at least $k$ of them. Thus, given a simple and connected graph $G = (V,E)$, a set $S \subset V$ is called a $k$-\emph{metric generator} for $G$ if and only
if any pair of different vertices of $G$ is distinguished by at
least $k$ elements of $S$, i.e., for any pair of different vertices
$u,v \in V$, there exist at least $k$ vertices $w_1,w_2, \dots,w_k \in  S$
such that 
\[d_G(u,w_i) \neq d_G(v,w_i), \mbox{ for every } i \in \{1,\dots,k\}.\]
A $k$-\emph{metric basis} is a $k$-metric generator of the minimum cardinality in $G$. Notice that if $k=1$ we obtain the classical definitions of metric basis and metric generator. 
Finally, $G$ is said to be a $k$-\emph{metric dimensional graph} if $k$ is the largest integer such
that there exists a $k$-metric basis for $G$. Let us denote $Dim(G)=k$ if $G$ is a $k$-metric dimensional graph. Notice that \[Dim(G)\geq 2 \ \forall \, G.\]

In \cite{ERY}, the authors provide several bounds on $Dim(G)$ and give some precise results in the case of trees. Herein, we provide some new bounds for this invariant and generalize some of their results. In particular, we extend their study for the case of trees to the case of clique graphs obtaining natural generalizations.

\section{Bounds on $Dim(G)$}

Given two different vertices $x,y \in V(G)$, the set of \emph{distinctive vertices} of $x,y$ is
$$\mathcal{D}_G(x,y) = \{z \in V(G) \, : \, d_G(x, z) \neq d_G(y, z)\},$$

\begin{theorem}\cite[Th. 1]{ERY}\label{th:distinctive} A connected graph G is $k$-metric dimensional
if and only if $$k = \min_{x,y\in V(G), \ x\neq y} |\mathcal{D}_G(x,y)|.$$
\end{theorem}

Let us recall the following definitions from \cite{ERY}. 
A vertex of degree at least
three in a graph $G$ will be called a major vertex of G. 
Any end-vertex (a vertex of degree one) $u$ of G is said to be a
terminal vertex of a major vertex $v$ of $G$ if
$d_G(u,v) < d_G(u,w)$ for every other major vertex $w$ of $G$.
The terminal degree $ter(v)$ of a major vertex $v$ is the
number of terminal vertices of $v$. 
Let $\mathcal{M}(G)$ be the set of exterior major vertices of
G having terminal degree greater than one.

Given $w \in \mathcal{M}(G)$ and a terminal vertex $u_j$ of $w$, let $P(u_j ,w)$ 
denote the shortest path that starts at $u_j$ and
ends at $w$. %Let $d(u_j ,w)$ be the length of $P(u_j ,w)$. 
Now, given $w \in \mathcal{M}(G)$ and two terminal vertices $u_j ,u_r$ of $w$ let $P(u_j,w,u_r)$ denote the shortest path from $u_j$ to $u_r$ 
containing $w$, and by $\varsigma(u_j ,u_r)$ the length of $P(u_j,w,u_r)$.
Notice that, by definition of exterior major vertex,
$P(u_j,w,u_r)$ is obtained by concatenating the paths
$P(u_j,w)$ and $P(u_r,w)$, where $w$ is the only vertex of
degree greater than two lying on these paths.
Finally, given $w\in \mathcal{M}(G)$ and the set of terminal
vertices $U = \{u_1,u_2,\dots,u_k\}$ of $w$, for $j \neq r$ let 
$\varsigma (w) = \min_{u_j ,u_r \in U} {\varsigma(u_j ,u_r)}$ and 
$\varsigma (G) = \min_{w\in \mathcal{M}(G)}{\varsigma(w)}$.  
%and $l(w) = \min_{u_j\in U}{l(u_j,w)}$.

\begin{theorem}\cite[Th. 3]{ERY}\label{Th:sigma} Let $G$ be a connected graph such that
$\mathcal{M}(G) \neq \emptyset$. Then, $Dim(G) \leq \varsigma(G)$.
\end{theorem}

However, this approach does not provide good bounds in many situations.

\begin{example} Suppose there exist two adjacent vertices $v_1,v_2$ with degree 2 such that both of them are adjacent to the same vertex $w$. Then, $w$ is a cut set and it is immediate to see that for any vertex $v'$ different from $v_1,v_2,w$ $d(v_1,v')=d(v_2,v')$. Thus, by Theorem \ref{th:distinctive}, $Dim(G)=2$, independently of the existence of terminal vertices or the value of $\varsigma(G)$.
\end{example}

Given any $v\in V$ and any $m\in \NN$, let $N(v,m)=\{w\in V \, : \, d(v,w)\leq m\}$, $S(v,m)=\{w\in V \, : \, d(v,w)= m\}$ and $\partial N(v,m)=\{w\in S(v,m) \, | \, d(w,G\setminus N(v,m))=1\}$.

\begin{definition} Given two different vertices $v,v'$ we say that they have \emph{equal $m$-boundary} if $\partial N(v,m)= \partial N(v',m)\neq \emptyset$.
\end{definition}

\begin{lemma}\label{l:boundary} Let $G$ be a connected graph. If two vertices, $v,v'$, in $G$   have equal $m$-boundary, then for every vertex $w$ in $G\setminus \Big( N(v,m)\cup N(v',m)\Big)$, $d(v,w)=d(v',w)$.
\end{lemma} 

\begin{proof} Consider any vertex $w$ in $G\setminus \Big( N(v,m)\cup N(v',m)\Big)$, and suppose $d(v,w)=p<q=d(v',w)$. Let $\gamma$ be a minimal path of length $p$ joining $v$ and $w$ and let $x$ be the vertex in $\gamma$ such that $d_G(v,x)=m$ and $d_G(x,w)=p-m$. Since $\partial N(v,m)= \partial N(v',m)$, then $d_G(v',x)=m$ and $d_G(v',w)\leq d_G(v',x)+d_G(x,w)=p<q$ leading to contradiction.
\end{proof}

Let $\mathcal{N}_m(G)\subset V\times V$ be the set of pairs of different vertices, $(v,v')$,  with equal $m$-boundary and $\mathcal{N}(G)=\cup_{m\in \NN}\mathcal{N}_m(G)$. For any $(v,v')\in \mathcal{N}_m(G)$ let $\eta_m(v,v')=|\{w\in N(v,m)\cup N(v',m) \, : \, d(v,w)\neq d(v',w)\}| \leq |N(v,m)\cup N(v',m)|$, $\eta_m (G)=\min_{(v,v')\in \mathcal{N}_m(G)} \eta_m(v,v')$ and $\eta (G)=\min_{m\in \NN} \eta_m (G)$.

\begin{theorem}\label{Th:diam} If $G$ is a connected graph such that
$\mathcal{N}_m(G) = \emptyset$ for every $m$, then $Dim(G) \geq \left\lfloor \frac{diam(G)-2}{4}\right\rfloor$.
\end{theorem}

\begin{proof} If $\mathcal{N}_m(G) = \emptyset$ for every $m$, given any pair of vertices $v,v'$ and any $m\in \NN$ such that either $\partial N(v,m)\neq \emptyset$ or $\partial N(v',m)\neq \emptyset$, $\partial N(v,m)\neq \partial N(v',m)$. Thus, we may assume that there is some $w\in \partial N(v,m)\setminus \partial N(v',m)$. Therefore, either $d(v',w)\neq m$ or there is a vertex $w'$ adjacent to $w$  such that $d(w',v)=m+1$ and $d(w',v')\neq m+1$.
Thus, for each $m\leq \frac{diam(G)}{2}-1$ there is a distinctive vertex of $v,v'$ at distance either $m$ or $m+1$ from $v$ and $Dim(G)\geq \left\lfloor \frac{diam(G)-2}{4}\right\rfloor$. 
\end{proof}

\begin{theorem}\label{Th:eta} If $G$ be a connected graph such that
$\mathcal{N}_m(G) \neq \emptyset$ for some $m$, then $Dim(G) \leq \eta(G)$.
\end{theorem}

\begin{proof} Consider $(v,v')\in \mathcal{N}_m(G)$ such that 
$|\{w\in N(v,m)\cup N(v',m) \, : \, d(v,w)\neq d(v',w)\}|=\eta(G)=k$ with $W=\{w_1,\dots,w_k\}$ the set of vertices $w \in N(v,m)\cup N(v',m)$ such that $d(v,w)\neq d(v',w)$. By Lemma \ref{l:boundary}, for every in $G\setminus \Big( N(v,m)\cup N(v',m)\Big)$, $d(v,w)=d(v',w)$. Therefore, $\mathcal{D}_G(v,v')=\{w_1,\dots,w_k\}$ and, by Theorem \ref{th:distinctive}, $Dim(G)\leq k$.
\end{proof}

\begin{theorem}\label{Th:eta2} If $G$ be a connected graph such that
$\mathcal{N}_m(G) \neq \emptyset$ for some $m$ and $\eta(G)\leq \left\lfloor \frac{diam(G)-2}{4}\right\rfloor$, then $Dim(G) = \eta(G)$.
\end{theorem}

\begin{proof} Given any pair of vertices, $v,v'$, if there is some $m$ such that they have equal $m$-boundary, then $|\mathcal{D}_G(v,v')|\geq \eta_m(v,v')\geq \eta(G)$. If for every $m$, $v,v'$ do not have equal $m$-boundary then, as we saw in the proof of Theorem \ref{Th:diam}, $\mathcal{D}_G(v,v')\geq \left\lfloor \frac{diam(G)-2}{4}\right\rfloor$. Therefore, for every $v,v'\in V$, $|\mathcal{D}_G(v,v')|\geq \eta(G)$ and, by Theorem \ref{th:distinctive}, $ Dim(G) \geq \eta(G)$. Thus, by \ref{Th:eta}, $Dim(G) = \eta(G)$.
\end{proof}

A vertex separator set in a graph is a set of vertices that disconnects two vertices. 

\begin{definition} Given two different vertices $v,v'$ in a connected graph $G$, we say that their $m$-spheres have a \emph{common separating subset} if there is a set of vertices $S\subset S(v,m)\cap S(v',m)$ such that $S$ is a vertex separator in $G$ and there is a component of $G\setminus S$ not containing $v$ nor $v'$. 
\end{definition}

\begin{remark}\label{r:separating} If two different vertices $v,v'$ in a connected graph $G$ have equal $m$-boundary and $N(v,m)\cup N(v',m)\neq V(G)$, then $S=\partial N(v,m)= \partial N(v',m)$ is a common separating subset. 
\end{remark}

\begin{lemma}\label{l:boundary2} If given two vertices, $v,v'$, in a connected graph $G$ their $m$-spheres have a common separating subset $S$, then for every vertex $w$ in any component of $G\setminus S$ not containing $v$ or $v'$, $d(v,w)=d(v',w)$.
\end{lemma} 

\begin{proof} Consider any vertex $w$ in any component of $G\setminus S$ not containing $v$ or $v'$, and suppose $d(v,w)=p<q=d(v',w)$. Let $\gamma$ be a minimal path of length $p$ joining $v$ and $w$. Since $S\subset S(v,m)\cap S(v',m)$ and $S$ separates $w$ from $v$ and $v'$ then there is some $x\in S\cap \g$ such that $d_G(v,x)=m$ and $d_G(x,w)=p-m$. Therefore, $d_G(v',x)=m$ and $d_G(v',w)\leq d_G(v',x)+d_G(x,w)=p<q$ leading to contradiction.
\end{proof}

\begin{lemma}\label{l:union} Given two vertices, $v,v'$, in a connected graph and two common separating subsets $S_1,S_2$ in their $m$-spheres, $S_1\cup S_2$ is a common separating subset.
\end{lemma}

\begin{proof} It is immediate to see that $S_1\cup S_2$ is a vertex separator contained in $S(v,m)\cap S(v',m)$. Now, let $w_1$ be any vertex in a component of $G\setminus S_1$ not containing $v$ or $v'$. Then, $d(w_1,v),d(w_1,v')>m$ and $w_1$ is contained in some component of $G\setminus (S_1\cup S_2)$ not containing $v$ or $v'$. 
\end{proof}

Let $\mathcal{P}_m(G)\subset V\times V$ be the set of pairs of different  vertices, $(v,v')$,  with some common separating component in their $m$-spheres and $\mathcal{P}(G)=\cup_{m\in \NN}\mathcal{P}_m(G)$. Given two vertices, $v,v'\in \mathcal{P}_m(G)$ let us denote $S_m(v,v')$, $m\in \NN$, the union of common separating subsets in their $m$-spheres and $C_m^{j}$, $j\in J$, the components of $G\setminus S_m(v,v')$ not containing $v$ or $v'$.  For any $(v,v')\in \mathcal{P}_m(G)$ let $\mu_m(v,v')=|\{w\notin \cup_{j}  C_m^{j}\, : \, d(v,w)\neq d(v',w)\}|$, $\mu_m (G)=\min_{(v,v')\in \mathcal{P}_m(G)} \mu_m(v,v')$ and $\mu (G)=\min_{m\in \NN} \mu_m (G)$.

\begin{definition} Given two vertices, $v,v'$, in a connected graph $G$, $ S_m (v,v') $ is critical if $ \mu_m(v,v') = \mu (G) $ and $ m = \min \{ k\in\NN : \, \mu_k(v,v') = \mu (G) \} $.
\end{definition}

\begin{proposition}\label{p:eta_mu} Given a connected graph $G$, if there exist  $m\in \NN$ and $v,v'\in \mathcal{N}_m(G)$  such that $\eta(G)=\eta_m(v,v')$ with $N_m(v)\cup N_m(v')\neq V(G)$, then $\eta(G)\geq \mu (G)$. 
\end{proposition}

\begin{proof} By Remark \ref{r:separating} it is immediate to check that $(v,v')\in \mathcal{P}_m(G)$ and  $\eta(G)=\eta_m(v,v') \geq \mu_m(v,v')\geq \mu_m (G)\geq \mu(G)$.
\end{proof}

\begin{corollary}  Given a connected graph $G$, if there exist  $m\in \NN$ and $v,v'\in \mathcal{N}_m(G)$  such that $\eta(G)=\eta_m(v,v')$ with $4m<diam(G)$, then $\eta(G)\geq \mu (G)$.
\end{corollary}

\begin{theorem}\label{Th:mu} Let $G$ be a connected graph such that
$\mathcal{P}(G) \neq \emptyset$, then $Dim(G) \leq \mu(G)$.
\end{theorem}

\begin{proof} Consider $(v,v')\in \mathcal{P}_m(G)$ such that 
$|\{w\notin \cup_{j}  C_m^{j}\, : \, d(v,w)\neq d(v',w)\}|=\mu(G)=k$ with $W=\{w_1,\dots,w_k\}$ the set of vertices $w \in G\setminus \Big(\cup_{j}  C_m^{j}\Big)$ such that $d(v,w)\neq d(v',w)$. By Lemma \ref{l:boundary2}, for every $w \in \cup_{j}  C_m^{j}$, $d(v,w)=d(v',w)$. Therefore, $\mathcal{D}_G(v,v')=\{w_1,\dots,w_k\}$ and, by Theorem \ref{th:distinctive}, $Dim(G)\leq k$.
\end{proof}

\begin{proposition}  Given a connected graph $G$, if there exist  $m\in \NN$ and $v,v'\in \mathcal{N}_m(G)$ such that 
$\eta_m(v,v')=\eta(G)$ with $N_m(v)\cup N_m(v')\neq V(G)$ and $4\eta(G)+2<diam(G)$, then $Dim(G)= \mu (G)$.
\end{proposition}

\begin{proof} By Remark \ref{r:separating}, $\mathcal{P}(G)\neq \emptyset$. Then, the result follows from Theorem \ref{Th:eta2}, Proposition \ref{p:eta_mu} and Theorem \ref{Th:mu}.
\end{proof}

\begin{corollary}  Given a connected graph $G$, if $\mathcal{N}_m(G)\neq \emptyset$ and $ \max\{4m,4\eta(G)+2\}<diam(G)$, then $Dim(G)= \mu (G)$.
\end{corollary}

By Proposition \ref{p:eta_mu} it is immediate to see that Theorem \ref{Th:mu} improves Theorem \ref{Th:eta} if there exist  $m\in \NN$ and $v,v'\in \mathcal{N}_m(G)$  such that $\eta(G)=\eta_m(v,v')$ with $N_m(v)\cup N_m(v')\neq V(G)$. Moreover, let us check that Theorem \ref{Th:mu} also improves Theorem \ref{Th:sigma}. 

\begin{remark}\label{r:cut} Given a connected graph $G$, if there exist  $m\in \NN$ and $v,v'\in V(G)$ such that there is a cut vertex $c$ in $S(v,m)\cap S(v',m)$, then $\{c\}$ is a common separating subset if there is some component $C$ of $G\setminus \{c\}$ not containing $v$ or $v'$.
\end{remark}

\begin{proposition} Let $G$ be a connected graph such that
$\mathcal{M}(G) \neq \emptyset$. Then, $\mathcal{P}(G)\neq \emptyset$ and $\mu(G)\leq \varsigma(G)$.
\end{proposition} 

\begin{proof} Consider any vertex $w\in \mathcal{M}(G)$ such that $\varsigma (w) = \varsigma(G)=k$ and let $u_i,u_j$ two terminal vertices of $w$ such that $\varsigma(u_i ,u_j)=k$. Consider two vertices $w_i,w_j$ adjacent to $w$ such that $w_i$ is contained in the path $[u_i,w]$ and $w_j$ is contained in the path $[u_j,w]$. Notice that $deg(w_i),deg(w_j)\leq 2$. Then, by Remark \ref{r:cut} and since $deg(w)\geq 3$, it is readily seen that $w$ is a common separating subset of $S(w_i,1)$ and $S(w_j,1)$. Also, since $u_i,u_j$ are terminal vertices of $w$, the union of the components of $G\setminus \{w\}$ containing $u_i$ or $u_j$ is exactly the path $P(u_j,w,u_r)$ and $\mu_1(w_i,w_j)=k\geq \mu(G)$.
\end{proof}

Also, another lower bound for  $Dim(G)$ can be given as follows:

Let $A(G)$ be the length of the shortest maximal (i.e. not contained in a longer) geodesic in 
$G$. 

For any two vertices, $v,w$, let $\g(v,w)$ be any maximal geodesic containing $v$ and $w$  and let us denote its length by $|\g(v,w)|$. Then, 
$$A(G)=\inf_{v\neq w}|\g(v,w)|.$$

\begin{proposition}\label{p:A(G)} If $G$ be a connected graph then 
\[A(G)\leq Dim(G).\]
\end{proposition}

\begin{proof} For any pair of vertices $v,w$ there is at most one vertex $x$ in $\g(v,w)$ such that $d(x,v)=d(x,w)$ which is the middle point between $v$ and $w$ if $d(v,w)$ is even, and at least $|\g(v,w)|$ vertices $y_j$ such that $d(y_j,v)\neq d(y_j,w)$. Therefore, by Theorem 
\ref{th:distinctive}, it is immediate that $A(G) = \inf_{v\neq w}|\g(v,w)|\leq Dim(G)$. 
\end{proof}

%\begin{proposition} Let $G$ be a connected graph such that
%$\mathcal{M}(G) \neq \emptyset$. Then, $A(G) \leq \varsigma(G)$.
%\end{proposition}

\begin{remark} Given a connected graph $G$ such that $\mathcal{M}(G) \neq \emptyset$, by \ref{Th:sigma} and \ref{p:A(G)} 
\[A(G) \leq  Dim(G) \leq \varsigma(G).\]

Notice that for any connected graph $G$ such that $\mathcal{M}(G) \neq \emptyset$, given $w \in \mathcal{M}(G)$ and two terminal vertices $u_j ,u_r$ of $w$ then $P(u_j,w,u_r)$ is a maximal geodesic in $G$. Hence, if $|P(u_j,w,u_r)|=A(G)$ for some $u_j,w,u_r$, then  
\[A(G) =  Dim(G) = \varsigma(G).\]
\end{remark}

\begin{example} Consider the graph $G$ in Figure \ref{Fig:1}. 
\begin{figure}[h]
\centering
\includegraphics[scale=0.7]{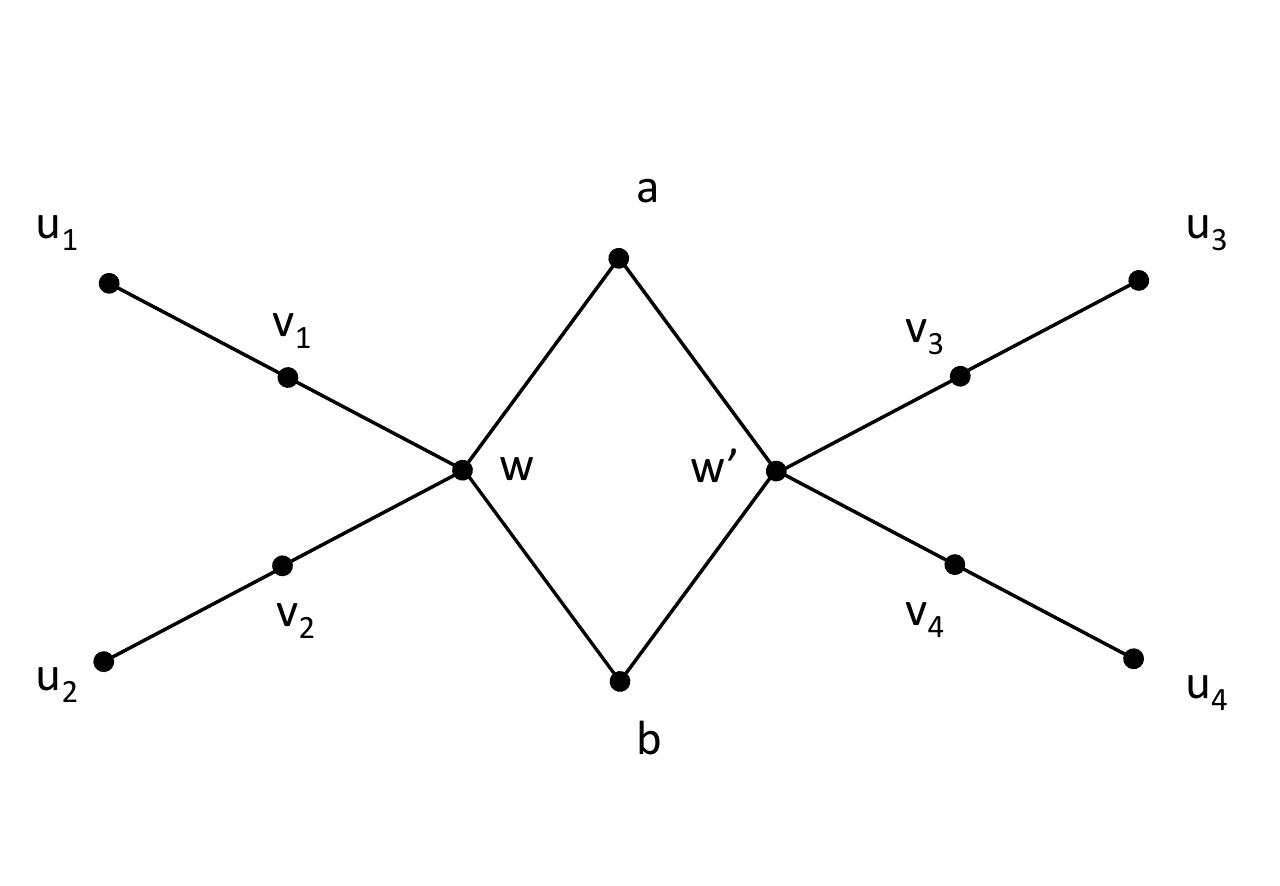}
\caption{For this graph, $2=A(G)=Dim(G)=\mu(G)=\eta(G)<\varsigma(G)=4$.}
\label{Fig:1}
\end{figure}
There are two major vertices, $w,w'$ in $G$ and each of them has two terminal vertices: $u_1,u_2$ are terminal vertices of $w$, $u_3,u_4$ are terminal vertices of $w'$. Thus, it is immediate to see that $\varsigma(w)=\varsigma(u_1,u_2)=4=\varsigma(u_3,u_4)=\varsigma(w)$ and $\varsigma(G)=4>Dim(G)$.

On the other hand, $\partial N(u_1,2)=\partial N (u_2,2)$, $\partial N(u_3,2)=\partial N (u_4,2)$, $\partial N(v_1,1)=\partial N (v_2,1)$, $\partial N(v_3,1)=\partial N (v_4,1)$ and  $\partial N(a,1)=\partial N (b,1)$. Hence, it is readily seen that $\eta_1(a,b)=\eta(G)=Dim(G)=2$. 

Also, $w$ is a common separating subset in $S(u_1,2)\cap S(u_2,2)$, in $S(v_1,1)\cap S(v_2,1)$, in $S(a,1)\cap S(b,1)$, in $S(u_3,4)\cap S(u_4,4)$ and in $S(v_3,3)\cap S(v_4,3)$. Also, $\{w,w'\}$ is a common separating subset in $S(v_1,1)\cap S(v_2,1)$ with $S_m(a,b)=\{w,w'\}$. There are other common separting subsets, however it is immediate to check that $\mu_1(a,b)=mu(G)=Dim(G)=2$.  

Finally, there is a maximal geodesic $[ab]$ with length 2 joining $a$ to $b$ and $A(G)=|[ab]|=Dim(G)=2$.

Thus, $2=A(G)=Dim(G)=\mu(G)=\eta(G)<\varsigma(G)=4$. 
\end{example}

\section{Block graphs}

A \emph{block graph} or \emph{clique tree} is a graph such that every biconnected component (\emph{block}) is a complete subgraph. In \cite{ERY}, the authors call it \textit{generalized tree} and define it using the following characterization.

Let $ \mathfrak{F} $ be the family of sequences  of connected graphs $ \mathcal{S} = ( G_1, ... , G_t )$, $ t\geqslant 2 $ such that $ G_1 $ is a complete graph $ K_{n_1} $ on $ n_1 \geqslant 2 $ vertices and $ G_i $, $ i \geqslant 2  $, is obtained recursively from $ G_{i-1} $ by adding a complete graph $ K_{n_i} $, $ n_i \geqslant 2  $,  and identifying one vertex of $ G_{i-1} $ with one vertex of $ K_{n_i} $. A connected graph $ G $ is a \emph{block graph} (or \textit{generalized tree}) if and only if there exists a sequence $ \mathcal{S} = ( G_1, ... , G_t ) \in \mathfrak{F} $ such that $ G_t = G $. From now on, we keep the more common name of block graph.
% When we want to show the construction of the generalized tree, we will denote the sequence as $ \mathcal{S} = ( K_{n_1}, (G_2 , K_{n_2}) , ... , (G_i , K_{n_i}), ... , G ) $.

In \cite{ERY}, a vertex $v$ is called an \emph{extreme vertex} if the subgraph
induced by $N[v]$ is isomorphic to a complete graph. Then, Corollary 3 states that a block graph $ G $ is 2-metric dimensional if and only if $ G $ contains at least two extreme vertices adjacent to a common cut vertex. Unfortunately, this result is not completely true as the following example shows. Remark \ref{r:gen_tree} and Proposition \ref{p:gen_tree} below give a characterization of 2-metric dimensional block graphs.

\begin{example}\label{Counter}
	Consider the graph $ G $ in Figure \ref{Fig:2}. There are two extreme vertices (non-adjacent) $ c $ and $ e $ in $ G $ being adjacent to a common cut vertex $ d $. However $ G $ is 3-metric dimensional as we can see in Table 1 below. 
	\begin{figure}[h]
		\centering
		\includegraphics[scale=0.9]{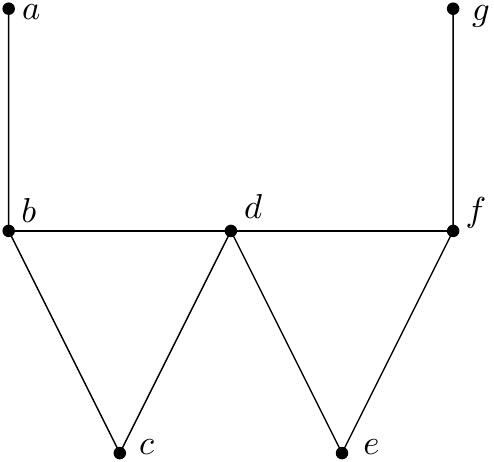}
		\caption{Counterexample to Corollary 3 in \cite{ERY}: $ G $ contains two extreme vertices adjacent to a common cut vertex and $ G $ is 3-metric dimensional. }
		\label{Fig:2}
	\end{figure}
	\begin{table}[h]
		\begin{center}
			$$ \begin{array}{|c|c|} 
			\hline
			x,y & \mathcal{D}_G (x,y) \smallsetminus \{ x,y \} \\ \hline
			a,b & \{c,d,e,f,g\} \\ \hline
			a,c & \{d,e,f,g\} \\ \hline
			a,d & \{c,e,f,g\} \\ \hline
			a,e & \{b,d,f,g\} \\ \hline
			a,f & \{b,d,e,g\} \\ \hline
			a,g & \{b,c,c,f\} \\ \hline
			b,c & \{a\} \\ \hline
			b,d & \{a,e,f,g\} \\ \hline
			b,e & \{a,c,f,g\} \\ \hline
			b,f & \{a,c,e,g\} \\ \hline
			b,g & \{a,c,d,f\} \\ \hline
			\end{array} \;\;
			\begin{array}{|c|c|} \hline
			x,y & \mathcal{D}_G (x,y) \smallsetminus \{ x,y \} \\ \hline
			c,d & \{e,f,g\} \\ \hline
			c,e & \{a,b,f,g\} \\ \hline
			c,f & \{a,b,e,g\} \\ \hline
			c,g & \{a,b,d,f\} \\ \hline
			d,e & \{a,b,c\} \\ \hline
			d,f & \{a,b,c,g\} \\ \hline
			d,g & \{a,b,c,e\} \\ \hline
			e,f & \{g\} \\ \hline
			e,g & \{a,b,c,d\} \\ \hline
			f,g & \{a,b,c,d,e\} \\ \hline
			& \\ \hline
			\end{array}  $$
			\caption{In the graph from Example \ref{Counter} (see Figure \ref{Fig:2}), it is immediate to check that $ \min_{x,y\in V(G), \ x\neq y} |\mathcal{D}_G(x,y)| = 3 = Dim(G)$.}
			\label{Tab:1}
		\end{center}
	\end{table}
	
\end{example}

\begin{definition} Given a graph $ G $ we say $ G $ is a V-graph if there exists a terminal vertex $ w $ such that $ \text{ter} (w) = 2 $ and $ d(u_i , w) = 1 $ for both terminal vertices $ u_i $ of $ w $.
\end{definition}

\begin{remark}\label{r:gen_tree} If $ G $ is a V-graph then $ G $ is 2-metric dimensional.
\end{remark}

By a cycle in a graph we mean a simple closed curve, this is, a path defined by a sequence of vertices which are all different except for the first one and last one which are the same.

\begin{remark}\label{r:cycle} If $T$ is a block graph, then every cycle in $T$ is contained in some complete subgraph.
\end{remark}

\begin{definition} We say that a block graph $G$ is \emph{tagged} if there is a maximal complete subgraph $K_r$ in $G$ with $r\geq 3$ and two vertices $u,v\in K_r$ such that $deg(u)=r-1=deg(v)$. %If there is not such a subgraph we say that $G$ is.
\end{definition}

\begin{proposition}\label{p:gen_tree} Consider $ G $ a block graph which is not a V-graph. Then $ G $ is 2-metric dimensional if and only if $G$ is tagged.
%there exists $ K_{n_i} $ with $ n_i \geqslant 3 $ such that $K_{n_i}$ contains two vertices of degree $ n_i - 1 $ in $ G $.
\end{proposition}

\begin{proof} Suppose that exists $ K_{r} $ with $ r \geqslant 3 $ and $u,v\in K_{r}$ with $ deg(u)=r-1=deg(v) $. Consider any other vertex $ x\in G $. If $ x \in K_{r} $ then $ d_G (x,u) = 1 = d_G (x,v) $ and $x \notin  \mathcal{D}_G (u,v)$. If $ x \notin K_{r} $ then, since $deg(u)=r-1=deg(v)$, there exists a vertex $ w \in K_{r} $, different from $ u $ and $ v $, such that $ d_G (x,w) = \min\limits_{y \in K_{r}} d_G (x,y) $. Therefore $ d_G (x,u) = d_G (x,w) + 1 = d_G (x,v) $ and $x \notin  \mathcal{D}_G (u,v)$. Thus, $ \mathcal{D}_G (u,v) = \{ u, v \} $ and, by Theorem \ref{th:distinctive}, $ G $ is 2-metric dimensional.

If $ G $ is 2-metric dimensional, by Theorem \ref{th:distinctive}, there exist two vertices $ u, v \in G $ such that $ \left| \mathcal{D}_G (u,v) \right|  = 2 $. 

Suppose that $ u \in K_{r} $ and $ v \in K_{s} $ with $K_r,K_s$ maximal complete subgraphs and $ K_r \neq K_s $. If $ d_G(u,v) = 1 $ then there exists a vertex $ u' \in K_{r} $ (or $ v' \in K_{s} $) such that $ \{u',u,v \}$ (or $ \{u,v,v' \}$) defines a geodesic path with length 2 and $ \left| \mathcal{D}_G (u,v) \right| \geqslant 3  $.
If $ d_G (u,v) = 2 $, let $ w $ be the vertex such that $ d_G(u,w) = 1 = d_G (v,w) $. Since $ G $ is not a V-graph, either $ u $ or $ v $ has an adjacent vertex different from $ w $. Suppose $ u' $ is adjacent to $ u $. Hence, by Remark \ref{r:cycle}, $u'$ is not adjacent to $v$ and therefore,  $ \{u',u,v \}\subset \mathcal{D}_G (u,v)$ and $ \left| \mathcal{D}_G (u,v) \right| \geqslant 3  $.
If $ d_G (u,v) \geqslant 3 $ it is trivial that $ \left| \mathcal{D}_G (a,b) \right| \geqslant 4  $. Thus, we conclude that $K_r=K_s$. 

If $deg(u)\geqslant r$ (respectively $deg(v)\geqslant r $), then there is a vertex $u'$ which is adjacent to $u$ and not adjacent to $v$. Thus, $ \{u',u,v \}\subset \mathcal{D}_G (u,v)$ and $ \left| \mathcal{D}_G (u,v) \right| \geqslant 3  $, leading to contradiction. 
\end{proof}

%\begin{definition} Given an infinite tree $ T $, we say $ T $ is a narrow tree if there exists a major vertex $ m \in T $ and two vertices $ u,v \in T $ such that $ \{u,m,v\} $ are in a geodesic path, $ d_G (u,m) = d_G (v,m) $ and $ C^m _u \cup C^m _v $ is finite, where $ C^y _x $ is the connected component of $ x $ in $ T\smallsetminus\{y\} $ for any vertices $ x \neq y $.\end{definition}

\begin{remark} Any finite graph is $k$-metric dimensional for some finite $k$.
\end{remark}

\begin{theorem} An infinite tree $ T $ is k-metric dimensional for some finite $k$ if and only if there exists a vertex $ w \in G $ such that $ T \smallsetminus \{ w \} $ has at least two finite connected components.
\end{theorem}

\begin{proof} Suppose $ T $ is $ k $-metric dimensional for some finite $ k $. By Theorem \ref{th:distinctive}, there exist two vertices $ u,v \in T $ such that $ \mathcal{D}_G(u,v) $ is finite. Consider $m$ the middle point in the geodesic path $[uv] $. If $ m $ is not a vertex, then $ \mathcal{D}_G(u,v) = T $, leading to a contradiction. If $ m $ is a vertex, then let $ C^m _u , C^m _v $ be the connected components of $T \smallsetminus \{ m \} $ containing $ u,v $ respectively. Therefore $ \mathcal{D}_G(u,v) = C ^m_u \cup C^m _v $ and $ C ^m_u \cup C^m _v $ is finite.
	
Suppose $ w \in G $ is a vertex such that $ T \smallsetminus \{ w \} $ has at least two finite connected components $C_1 , C_2$. Consider $ v_i \in C_i $ with $ d_G(v_i , w) = 1 $ for $ i = 1,2 $. Then $ \mathcal{D}_G (v_1,v_2) = C_1 \cup C_2 $ and $ \mathcal{D}_G (v_1,v_2) $ is finite. 
\end{proof}

\begin{definition} An infinite block graph $G$ is narrow if one of the following conditions holds:
	\begin{enumerate}
		\item[(i)] There exists a vertex $w\in G$ such that $G\smallsetminus \{ w\} $ has at least two finite connected components.
		
		\item[(ii)] There exists a complete subgraph $ K_n $, $ n\geqslant 3 $, such that $ G \smallsetminus E(K_n) $ has at least two finite connected components.
	\end{enumerate}
\end{definition}

\begin{remark}\label{r:unique} It is well known that if $ G $ is a block graph then given any two vertices in $G$ there is a unique geodesic path joining them.
\end{remark}

\begin{theorem} An infinite block graph $ G $ is $k$-metric dimensional for some finite $k$ if and only if $ G $ is narrow.
\end{theorem}

\begin{proof} Suppose $G$ is $k$-metric dimensional for some finite $k$. By Theorem \ref{th:distinctive}, there exist two vertices $u,v \in G$ such that $ \mathcal{D}_G(u,v) $ is finite. Consider $m$ the middle point in the geodesic path $[uv]$. If $m$ is a vertex, then let $ C^m _u , C^m _v $ be the connected components of $G \smallsetminus \{ m \} $ containing $ u,v $ respectively. Therefore $ \mathcal{D}_G(u,v) = C ^m_u \cup C^m _v $ and $ C ^m_u, C^m _v $ are finite.
If $m$ is not a vertex, consider the edge $ e \in T $ such that $m \in e$. We can see that $ e \in K_n $ for some $n\geqslant 3$ since otherwise $ \mathcal{D}_G(u,v) = T $, leading to a contradiction. Then let $ C_u , C_v $ be the connected components of $T \smallsetminus E(K_n) $ containing $ u,v $ respectively. Therefore $ \mathcal{D}_G(u,v) = C_u \cup C_v  $ and $ C_u, C_v $ are finite.

Suppose $ T $ is narrow. If $(i)$ holds, there exist a vertex $ w \in T $ such that $ T \smallsetminus \{ w \} $ has at least two finite connected components $C_1 , C_2$. Consider $ v_i \in C_i $ with $ d_G(v_i , w) = 1 $ for $ i = 1,2 $. Then $ \mathcal{D}_G (v_1,v_2) = C_1 \cup C_2 $ and $ \mathcal{D}_G (v_1,v_2) $ is finite. If $(ii)$ holds, there exists a complete subgraph $ K_n $, $ n\geqslant 3 $, such that $ T \smallsetminus E(K_n) $ has at least two finite connected components $ C_1 , C_2 $. Consider $ u_i \in K_n \cap C_i $ for $i=1,2$. Therefore $ \mathcal{D}_G (u_1,u_2) = C_1 \cup C_2 $ and $ \mathcal{D}_G (u_1,u_2) $ is finite.
\end{proof}

Let us recall the following result on $k$-metric dimensional trees.

\begin{theorem}\cite[Th. 9]{ERY}\label{th:treedimension} If $T$ is a k-metric dimensional tree different from a path, then $k = \varsigma(T)$.
\end{theorem}

This result can be generalized for block graphs using $\mu(T)$ to obtain Theorem 
\ref{th:blockdimension} below.

\begin{definition} A block graph $G$ is \emph{non-elementary} if it is neither a complete graph nor a path graph.
\end{definition}

\begin{proposition}\label{P:non-elementary} If $G$ is a  block graph, then $ \mathcal{P}(G) \neq \emptyset $ if and only if $G$ is non-elementary.
\end{proposition}

\begin{proof} Suppose $G$ is a non-elementary block graph. Since $G$ is not a path, there exists a complete subgraph $K_n \subseteq T$ with $n\geqslant 3$. Since $G$ is not a complete graph, then there exists $v\in K_n$ such that $\delta (v) > n - 1$. Since $n\geqslant 3$, then there exist two different vertices $x,y \in K_n \smallsetminus \{v\}$. Therefore, $S_1(x,y)\neq\emptyset $ and $\mathcal{P}(G) \neq\emptyset$.

If $G$ is a complete graph or a path graph, then it is trivial to check that $\mathcal{P}(G) = \emptyset $.
\end{proof}

\begin{remark} For any $n\geq 3$,
\begin{itemize} 
\item if $K_n$ is the complete graph with $n$ vertices, then $Dim(K_n)=2$.
\item f $P_n$ is the path graph with $n$ vertices, then $Dim(P_n)=n-1$.
\end{itemize}
\end{remark}

\begin{proposition}\label{P:tagged} It $G$ is a non-elementary tagged block graph, then $\mu(G)=2$.
\end{proposition}

\begin{proof} Let $K_r$ be a maximal complete subgraph in $G$ with $r\geq 3$ and two vertices $u,v\in K_r$ such that $deg(u)=r-1=deg(v)$. Since $G$ is non-elementary, $G\neq K_r$ and there is a vertex $w\in K_r$ with $deg(w)\geq r$. Thus, $S_1(u,v)=K_r\setminus \{u,v\}$ and $\mu(G)=\mu_1(u,v)=2$.
\end{proof}

\begin{proposition}\label{P:non-tagged} It $G$ is a non-elementary non-tagged block graph, then for every pair of vertices $x,y\in G$ either $(x,y)\in \mathcal{P}(G)$ or $|\mathcal{D}_G(x,y)|\geq |G|-1$. 
\end{proposition}

\begin{proof} Given two vertices $x,y\in G$ we can distinguish two cases.

\textbf{Case 1:} If $k:=d_G(x,y)$ is even. Then, there is a vertex $w$ such that $d_G(x,w)=\frac{k}{2}=d_G(w,y)$. Then $G\setminus w$ has at least two connected components, $C_x,C_y$, containing $x$ and $y$ respectively and it is readily seen that $\forall \, v\in C_x$, $d_G(x,v)<d_G(v,y)$ and $\forall \, v\in C_y$, $d_G(x,v)>d_G(v,y)$. Therefore, if $T\setminus w=C_x\cup C_y$, then $|\mathcal{D}_G(x,y)|= |G|-1$. Otherwise, $w\in S_{k/2}(x,y)$ and $(x,y)\in \mathcal{P}(G)$.  

\textbf{Case 2:} If $k:=d_G(x,y)$ is odd. Then, there are two (adjacent) vertices, $x',y'\in [xy]$ 
 such that $d_G(x,x')=\frac{k-1}{2}=d_G(y',y)$ (with $x=x'$ and $y=y'$ if $k=1$). Let $K_r$ be the maximal complete subgraph containing $x'y'$. 

If there is a connected component $C$ in $G\setminus K_r$ which is not adjacent to $x'$ nor $y'$, then $r\geq 3$ and there is a vertex $w\in K_r\setminus \{x',y'\}$ such that $w\in S_{(k+1/2)}(x,y)$. Thus, $(x,y)\in \mathcal{P}(G)$. For every connected component $C_x$ of $G\setminus K_r$ adjacent to $x'$ and $\forall v\in C_x$, it is clear that $d_G(x,v)<d_G(v,y)$, and  for every connected component $C_y$ of $G\setminus K_r$ adjacent to $y'$ and $\forall v\in C_y$, it is clear that $d_G(x,v)<d_G(v,y)$. Thus, if every connected component of $G\setminus K_r$ is adjacent to $x'$ or $y'$, $\mathcal{D}_G(x,y)\subset K_r\setminus \{x',y'\}$. Thus, if  $r\leq 3$, $|\mathcal{D}_G(x,y)|\geq |G|-1$ and if $r>3$, $G$ is tagged leading to contradiction.
\end{proof}

\begin{theorem}\label{th:blockdimension} If $G$ is a non-elementary block graph, then $Dim(G) = \mu(G)$.
\end{theorem}

\begin{proof} By Proposition \ref{P:non-elementary}, $ \mathcal{P}(G) \neq \emptyset $. 

If $G$ is tagged, by propositions \ref{p:gen_tree} and \ref{P:tagged}, $Dim(G)=\mu(G)=2$.

If $G$ is not tagged, consider any pair of vertices $x,y\in G$. By Proposition \ref{P:non-tagged}, either $(x,y)\in \mathcal{P}_k(G)$ for some $k$ and $|\mathcal{D}_G(x,y)|=\mu_k(x,y)\geq \mu(G)$ or $|\mathcal{D}_G(x,y)|\geq |G|-1\geq \mu(G)$. Thus, $Dim(G)\geq \mu(G)$ and, by Theorem \ref{Th:mu}, $ Dim(G) =\mu(G) $. 
\end{proof}

The problem of computing $\mu(G)$ using the definition has relatively high computational complexity. To improve the interest of Theorem \ref{th:blockdimension}, this complexity can be reduced using some properties of common separating subsets in block graphs.

\begin{lemma}\label{l:complete} If $ G $ is a block graph and $ \mathcal{P}(G) \neq\emptyset $ then, for every vertices $ (v,v') \in \mathcal{P}(G)$ with $ S_m (v,v') $ critical, $ S_m (v,v') $ is complete.
\end{lemma}

\begin{proof} If $ S_m (v,v') $ is critical, let $ S_m (v,v')=\{c_i\}_{i\in I} $ and for every $i\in I$ let $u_i$ be the vertex such that $[vc_i]\cap [v'c_i]=[u_ic_i]$.
For any  $u_i\neq u_j$, let $x,y$ be the vertices such that $ [vu_i]\cap [vu_j] = [vx] $ and $ [v'u_i]\cap [v'u_j] = [v'y] $.

Claim: $ [xu_i]\cup [u_iy]\cup [yu_j]\cup [u_jx] $ defines a cycle.

Suppose $ [u_i y]\cap [u_j x] \neq\emptyset $ (if $ [u_i x]\cap [u_j y] \neq\emptyset $ the same argument holds). Consider any vertex $w \in [u_i y]\cap [u_j x] $, then let $ d_G(v,w) = a ,\, d_G(w,c_j) = b ,\, d_G(v',w) = a' ,\, d_G(w,c_i) = b $ (See Figure \ref{Fig:cycle}). 
	\begin{figure}[h]
		\centering
		\includegraphics[scale=0.5]{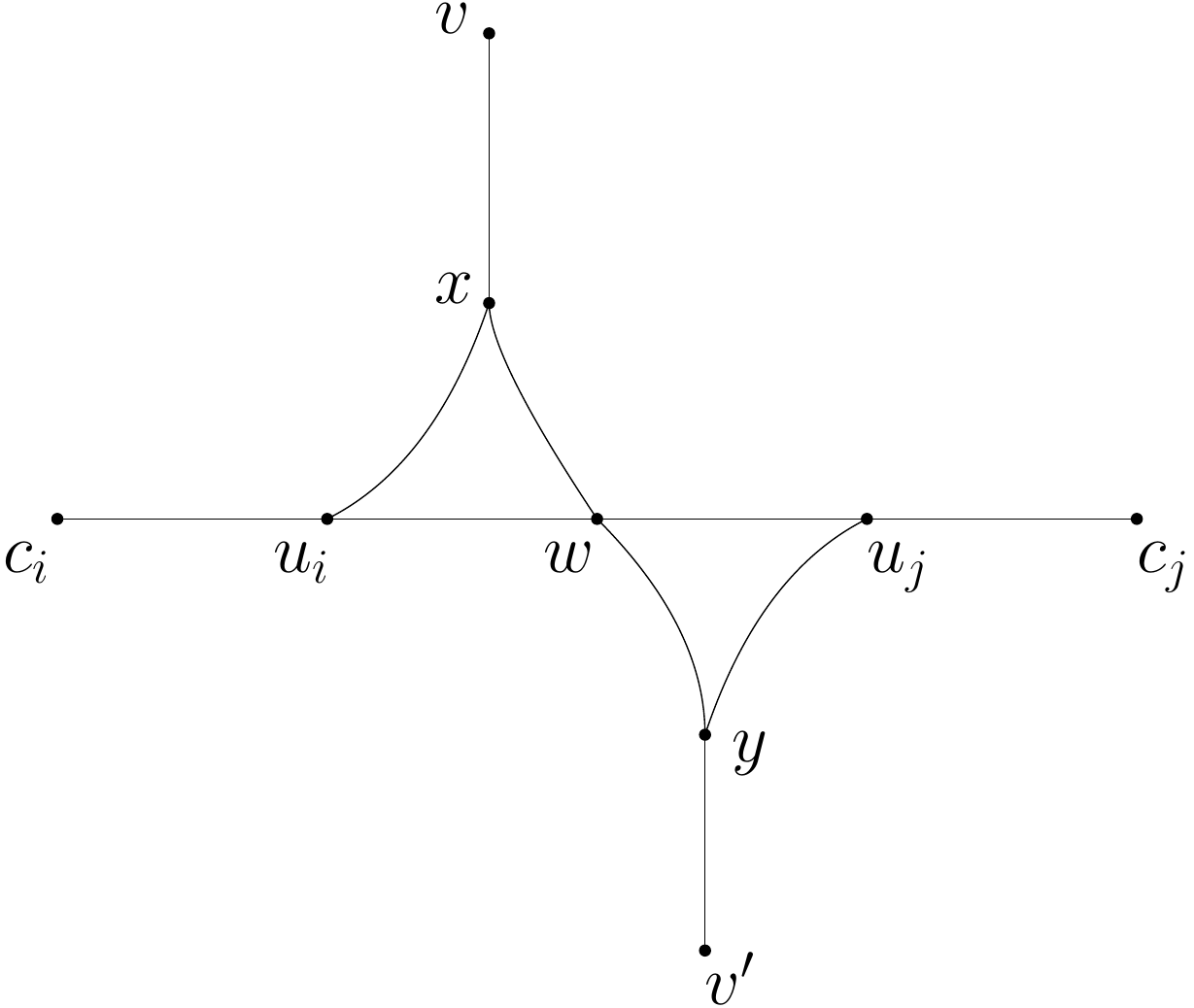}
		\caption{Since geodesics are unique, $ [xu_i]\cup [u_iy]\cup [yu_j]\cup [u_jx] $ defines a cycle.}
		\label{Fig:cycle}
	\end{figure}
Since $ m = a' + b' = a + b $ and by Remark \ref{r:unique} geodesics are unique, therefore
$$ \begin{array}{ccc}
a' + b' < a + b' & \Longrightarrow & a' < a \\ 
a + b < a' + b & \Longrightarrow & a < a'
\end{array}  $$
leading to contradiction and proving the claim.

By Remark \ref{r:cycle}, $ u_i ,\, u_j $ are adjacent and therefore, $ \{u_i \}_{i\in I} $ induces a complete subgraph.

Since  $d_G(v,c_i)=d_G(v',c_i)$, then $d_G(v,u_i)=d_G(v',u_i)$ $\forall i$. Moreover, we can see that $d_G(v,u_i)=d_G(v,u_j)$ for every $i\neq j$. Suppose $d_G(v,u_i)<d_G(v,u_j)$ for some $i\neq j$. Then, since $ \{u_i \}_{i\in I} $ induces a complete subgraph and $d_G(v,u_i)=d_G(v',u_i)$ $\forall i$, $d_G(v,u_i)+1=d_G(v',u_i)+1=d_G(v',u_j)=d_G(v,u_j)$. Hence, since geodesics in $G$ are unique, 
$u_i\in [vu_j]\cap [v'u_j]$ leading to contradiction.

Suppose $ r = d_G(v,u_i) < m $ for every $i\in I$. Then,  $ \{u_i \}_{i\in I}\subset S_r(v,v') $ with $r<m$ and $S_m(v,v')$ is not critical leading to contradiction. Therefore,  $ u_i=c_i \, \forall  i\in I $ and $ S_m(v,v')=\{c_i\}_{i\in I} $  is complete.
\end{proof}

\begin{remark}\label{r:maxcomplete} Suppose $G$ is a block graph and 
$ (v,v') \in \mathcal{P}(G) $ with $ S_m (v,v') $ critical. If $| S_m(v,v') | > 1$, $K$ is the complete maximal subgraph such that $ S_m (v,v') \subseteq K $ and $v,v' \notin K$, the following properties hold.
	\begin{enumerate}
		\item[(i)] $v,v'$ are in different connected components of $G\smallsetminus K$.
		\item[(ii)] $ |K| = |S_m (v,v')| + 2 $.
		\item[(iii)]  there are at least there vertices $x_1,x_2,x_3 \in K$ such that $deg_G(x_i)\geq |K|$.
	\end{enumerate}
\end{remark}

\begin{proposition}\label{p:mu1} If $G$ is a block graph, then \[\mu (G)=\mu_1 (G).\]
\end{proposition}

\begin{proof} Suppose $\mu (G) = \mu _m (v,v')$ for some $v,v'\in G$ and $m>0$ such that $ S_m(v,v') $ is critical. By Lemma \ref{l:complete}, $ S_m(v,v') $ is complete. Let $K$ be a complete maximal subgraph such that $ S_m (v,v') \subseteq K $. 

If $ | S_m(v,v') | > 1 $ and $v,v' \in K$, then $ m = 1 $ and $ \mu (G) = \mu _1 (v,v') = \mu _1 (G) $. If $ | S_m(v,v') | > 1 $ and $v,v' \notin K$, then by remark \ref{r:maxcomplete}, $v,v'$ are in different connected components $ C_v , C_{v'} $ of $G\smallsetminus K$ respectively. % and $K$ is a 3-cut block. 
Let $x,y \in K$ be the vertices adjacents to $C_v , C_{v'}$ resp., then $ d_G(x,z) = d_G(y,z) = 1 $ for every $z \in S_m(v,v') = K \smallsetminus \{x,y\}$ and $ S_1(x,y) = S_m(v,v') $. Therefore $ \mu (G) = \mu _m (v,v') = \mu _1 (x,y) = \mu _1 (G) $. 

Suppose $S_m(v,v') = w$ and let $x, y$ be the vertices adjacent to $w$ in the geodesic paths $[vw]$ and $[v'w]$ respectively. Obviously $ S_1(x,y) = w = S_m(v,v') $ and $ \mu (G) = \mu _m (v,v') = \mu _1 (x,y) = \mu _1 (G) $.
\end{proof}

\begin{definition} If $G$ is a non-elementary block graph, then
\begin{itemize}
	\item a vertex $w\in G$ is a \emph{3-cut vertex} if $w$ is a cut vertex and $deg_G(w)\geq 3$, %$T\setminus w$ has at least 3 connected components,
	\item a maximal complete subgraph $K\subset G$ is a \emph{3-cut block} if $|K|\geq 3$, 
	%and there are at least two vertices $x_1,x_2 \in K$ such that $deg_G(x_i)\geq |K|$,
	\item a \emph{3-cut piece} is either a 3-cut vertex or a 3-cut block.
\end{itemize}
\end{definition}

\begin{remark} If $G$ is a non-elementary block graph and $w$ is a 3-cut vertex, then there are three vertices $v_1,v_2,v_3\in S(w,1)$ such that if $C_i$ is the connected component of $T\setminus w$ containing $v_i$, then $C_1\cup C_2$ and $C_3$ are disjoint. Thus, $(v_1,v_2)\in \mathcal{P}_1(G)$, $S_1(v_1,v_2)=w$ and $\mu_1(v_1,v_2)=|C_1\cup C_2|$.

If $G$ is a non-elementary block graph and $K$ is a 3-cut block, then there are at least three vertices $v_1,v_2,v_3\in K$ and one of them, suppose it is $v_3$ satisfies that $deg_G(v_3)\geq |K|$. Thus, $(v_1,v_2)\in \mathcal{P}_1(G)$, $S_1(v_1,v_2)=K\setminus\{v_1,v_2\}$ and $\mu_1(v_1,v_2)=|C_1\cup C_2|$ where $C_i$ is the connected component of $G\setminus E(K)$ containing $v_i$.
\end{remark}

\begin{definition} If $G$ is a non-elementary block graph, then 
\begin{itemize}
	\item a 3-cut vertex $w\in G$ is \emph{extremal} if there are two vertices $v,v'$ adjacent to $w$ such that $S_1(v,v')=w$ and the connected components $C_v,C_{v'}$ of $T\setminus w$ containing $v,v'$ respectively (where possibly $C_v=C_{v'}$ if $v,v'$ are adjacent) do not contain any 3-cut piece,
	\item a 3-cut block $K$ is \emph{extremal} if 
	%at least two components of $T\setminus E(K)$ 
	there are two vertices $v,v'\in K$ such that $S_1(v,v')\subset K$ and 
		%with $\max\{deg(v),deg(v')\}\geq |K|$ and 
	the connected components $C_v,C_{v'}$ of $T\setminus E(K)$ containing $v,v'$ respectively do not contain any 3-cut piece, 
	%and some vertex $w\in K\setminus\{v,v'\}$ with $deg(w)\geq \K|$,
	\item a \emph{3-cut piece} is \emph{extremal} it if is either an extremal 3-cut vertex or an extremal 3-cut block.
\end{itemize}
\end{definition}

Given a block graph $G$, let $\mathcal{E}(G)\subset V\times V$ be the set of pairs of different vertices, $(v,v')$, such that one of the following conditions holds: 
\begin{itemize}
	\item $S_1(v,v')=w$ with $w$ a 3-cut vertex and the components $C_v,C_{v'}$ of $T\setminus w$ containing $v,v'$ respectively, do not contain a 3-cut piece,
	\item $v,v',S_1(v,v')\subset K$ with $K$ a 3-cut block 
	%$\max\{deg(v),deg(v')\}\geq |K|$ 
	and the components $C_v,C_{v'}$ of $T\setminus E(K)$ containing $v,v'$ respectively, do not contain a 3-cut piece.
\end{itemize}

\begin{remark} Notice that if $T$ is a tree, the extremal 3-cut pieces are exactly the major vertices. Also, $\mathcal{E}(G)\subset \mathcal{P}_1(G)$.
\end{remark}

Given a 3-cut piece $P$, let us denote $G\setminus [P]$ the set $G\setminus P$ if $P$ is a cut vertex or $G\setminus E(P)$ if $P$ is a cut block.

\begin{theorem}\label{th:mu1} If $G$ is a non-elementary block graph with $\mu (G)<\infty$, then 
\[\mu (G)=\min_{(v,v')\in \mathcal{E}(G)} \mu_1(v,v').\] 
\end{theorem}

\begin{proof} By Proposition \ref{p:mu1}, $ \mu(G) = \mu_1(G) = \mu_1(v,v') $ for some  vertices $v,v' \in G$. Then, $S_1(v,v')$ is contained in some 3-cut piece $P$. Let $C_1,C_2$ be the connected components of $G\smallsetminus [P]$ containing $v,v'$ respectively. Hence, $\mu_1(v,v')=|C_1\cup C_2|$.

Suppose $ (v,v') \notin \mathcal{E}(G) $. Then, either $C_1$ or $C_2$ contains a 3-cut piece. Suppose without loss of generality that $ C_1 $ contains a 3-cut piece. Since  $ |C_1| \leqslant \mu(G) $, $C_1$ is finite and therefore $C_1$ contains an extremal 3-cut piece $P'$. Then, there are two vertices $(x,y) \in  \mathcal{E}(G) $ such that $S_1(x,y)\subset P'$ and such that the connected components, $C'_1,C'_2$, of $T\setminus [P']$ containing $x,y$ respectively, do not contain any 3-cut piece. Hence, in particular, $C'_1,C'_2$ do not contain $P$ and therefore $C'_1,C'_2\subset C_1\setminus S_1(x,y)$. Thus, $\mu_1(x,y)=|C'_1\cup C'_2|<|C_1|\leq \mu_1(v,v')$ leading to contradiction.
\end{proof}

\begin{corollary}  If $G$ is a non-elementary block graph with $\mu (G)<\infty$, then  
\[Dim(G)=\min_{(v,v')\in \mathcal{E}(G)} \mu_1(v,v').\] 
\end{corollary}

\end{document}